\documentclass[11pt]{article}
\usepackage{amssymb, amsmath}
\renewcommand{\theequation}{\arabic{equation}}

\usepackage[margin=1.2in]{geometry}

\usepackage{mathptmx}

\newtheorem{theorem}{Theorem}[section]
\newtheorem{lemma}[theorem]{Lemma}
\newtheorem{proposition}[theorem]{Proposition}
\newtheorem{corollary}[theorem]{Corollary}

\newenvironment{proof}[1][Proof]{\begin{trivlist}
\item[\hskip \labelsep {\bfseries #1}]}{\end{trivlist}}
\newenvironment{definition}[1][Definition]{\begin{trivlist}
\item[\hskip \labelsep {\bfseries #1}]}{\end{trivlist}}
\newenvironment{example}[1][Example]{\begin{trivlist}
\item[\hskip \labelsep {\bfseries #1}]}{\end{trivlist}}
\newenvironment{remark}[1][Remark]{\begin{trivlist}
\item[\hskip \labelsep {\bfseries #1}]}{\end{trivlist}}

\newcommand{\qed}{\nobreak \ifvmode \relax \else
      \ifdim\lastskip<1.5em \hskip-\lastskip
      \hskip1.5em plus0em minus0.5em \fi \nobreak
      \vrule height0.75em width0.5em depth0.25em\fi}

\begin{document}
\title{A Note on Moments of Limit Log Infinitely Divisible Stochastic Measures of Bacry and Muzy}

\author{Dmitry Ostrovsky}

\date{September 1, 2016}

\maketitle
\noindent
\begin{center}
{\small 195 Idlewood Drive, Stamford, CT 06905, USA; email:
dm\_ostrov@aya.yale.edu} \\
\end{center}

\begin{abstract}
\noindent A multiple integral representation of single and joint moments of the total mass of the limit log-infinitely divisible stochastic measure of Bacry and Muzy [{\it Comm. Math. Phys.} {\bf 236}: 449-475, 2003] is derived. The covariance structure of the total mass of the measure is shown to be logarithmic. 
A generalization of the Selberg integral corresponding to single moments of the limit measure is proposed and shown to satisfy a recurrence relation. 
The joint moments of the limit lognormal measure, classical Selberg integral with $\lambda_1=\lambda_2=0,$ and Morris integral are represented in the form of multiple binomial sums. 
For application, low moments of the limit log-Poisson measure are computed exactly
and low joint moments of the limit lognormal measure are considered  in detail.
\end{abstract}

\noindent {\bf Keywords:} Multifractal stochastic measure, multiplicative chaos, intermittency, Selberg integral, infinite divisibility, L\'evy-Khinchine decomposition, binomial sum, joint moments.

\noindent {\bf Mathematics Subject Classification (2010):} 60E07, 60G57, 05A10, 33F10.

\section{Introduction}
\noindent 
In this note we contribute to the study of limit log-infinitely divisible
(logID for short) random measures (also known as multiplicative chaos or cascades) on the unit interval. This study was initiated by Mandelbrot \cite{secondface}, \cite{Lan} and Bacry \emph{et. al.} \cite{MRW} in the limit lognormal case, extended to the compound Poisson case by Barral and Mandelbrot \cite{Pulses}, and developed in the general infinitely divisible case by Bacry and Muzy \cite{BM1}, \cite{BM}. 
The formal mathematical theory of multiplicative chaos was founded by Kahane
\cite{K2}. The interest in this class of measures derives from their remarkable property of stochastic self-similarity with log-infinitely divisible multipliers. In addition, they are grid-free and stationary so that their moments are exactly multiscaling. Since its formal inception in 2002, this class of multifractal random measures has generated a significant level of interest in mathematical physics, especially in the context of KPZ, cf. \cite{BenSch} and  \cite{RhoVar}. Of all such measures, the limit lognormal measure has enjoyed the greatest amount of attention, in part due to the connection of the moments of its total mass with the classical Selberg integral, cf. \cite{BDM}, \cite{FLDR}, \cite{Me4}, \cite{MeIMRN}, \cite{Me14}, its connection with quantum gravity \cite{DS}, and its conjectured relevance to the Riemann zeta function \cite{YK}, \cite{Me14}. The limit lognormal measure is also intimately related to the logarithmically correlated gaussian free field, which has recently attracted a great deal of interest, cf. \cite{YO} and references therein. We refer the reader to \cite{Me5} for a review of the original construction and to \cite{barralID} for recent extensions. 

In this note we investigate the single and joint moments of the total mass of the general limit logID measure.
Our interest in the moments has to do with the fact that our theory of intermittency expansions allows one to reconstruct the full distribution of the total mass from the moments, cf.
\cite{Me3}, \cite{Me5}. In the limit lognormal case we applied this theory to compute the Mellin transform of the total mass exactly, cf. \cite{Me4}, and then extended it to the joint distribution  of the mass of several intervals, cf. \cite{Me6}. The technique of intermittency expansions requires a closed-form formula for the moments of the total mass. The only known case of such a formula is that of the limit lognormal measure, whose moments are given by the classical Selberg integral with $\lambda_1=\lambda_2=0$ as was first pointed out in \cite{BDM}. The main contribution of this note is to represent both single and joint moments of the total mass of the general limit logID measure by novel multiple integrals that generalize the Selberg integral. While we do not know how to compute them in closed-form for arbitrary moments, we derive a general recurrence relation for them 
in cases of single moments and moments of two subintervals of the unit interval. In the case of joint moments of the limit lognormal measure, 
we show that the corresponding multiple integral can be represented in the form of a multiple binomial sum. We also give novel multiple binomial sum interpretations of the Selberg integral with $\lambda_1=\lambda_2=0$ and the Morris integral, complementing the recent combinatorial study in \cite{Petrov}. For application of our theoretical results, we treat low moments of the limit log-Poisson measure 
and the simplest nontrivial joint moments of the limit lognormal measure. 

Our results are exact except for the proof of Corollary \ref{cov}, which partially relies on a heuristic argument. The proofs of our binomial sum
results are only sketched for brevity. 

The plan of the paper is as follows. In Section 2 we give a brief review of the general limit logID measure and in Section 3 we state and prove our main results. In Section 4 we treat the limit log-Poisson measure and in Section 5 the joint moments of the limit lognormal measure. In Section 6 we give conclusions. 
\section{A review of limit logID measures on the unit interval}
\noindent In this section we will review the limit
log-infinitely divisible (logID) construction following \cite{BM1} and \cite{BM}, except for several notation-related
changes to be explained below. 
The starting point is an infinitely divisible (ID) independently scattered random measure
$P$ on the time-scale plane $\mathbb{H}_+=\{(t, \, l), \, \, l>0\},$
distributed uniformly with respect to the intensity measure $\rho$
(denoted by $\mu$ in \cite{BM})
\begin{equation}
\rho(dt \, dl)=dt\,dl/l^2.
\end{equation}
This means that $P(A)$ is ID for measurable subsets
$A\subset\mathbb{H}_+,$ $P(A)$ and $P(B)$ are independent if $A\bigcap B=\emptyset,$ 
and
\begin{equation}
{\bf{E}} \left[ e^{i q P(A)} \right]=e^{\mu\phi(q)\rho(A)},
\,\,\,q\in\mathbb{R},
\end{equation}
where $\mu>0$ is the intermittency parameter\footnote{What we call
$\mu$ is denoted $\lambda^2$ in \cite{BM}. Also, in \cite{BM}
it is taken to be part of $\phi(q),$ whereas we prefer to
separate the two.} and $\phi(q)$ is the logarithm of the characteristic
function of the underlying ID distribution and is given by the
L\'evy-Khinchine formula
\begin{equation}\label{phi}
\phi(q) = -\frac{iq\sigma^2}{2}  - \frac{q^2\sigma^2}{2} + 
\int\limits_{\mathbb{R}\setminus \{0\}} \Bigl(e^{iq u}-1-iq(e^u-1)
\Bigr) d\mathcal{M}(u).
\end{equation}
It is normalized by $\phi(-i)=0$ so that ${\bf E} \bigl[e^{P(A)}\bigr]=1$
for all measurable subsets $A\subset\mathbb{H}_+.$ The constant $\sigma$ satisfies 
$\sigma^2\geq 0$ and the spectral function $\mathcal{M}(u)$ is
continuous and non-decreasing on $(-\infty, 0)$ and $(0, \infty)$, and satisfies the integrability and limit
conditions $\int_{[-1,1] \setminus \{0\}} u^2
d\mathcal{M}(u)<\infty$ and $\lim\limits_{u\rightarrow\pm\infty}
\mathcal{M}(u)=0.$ We will further assume that $\mathcal{M}(u)$
decays at infinity fast enough so that all the integrals with respect to it
in this and next sections converge, 
which restricts the class of permissible spectral
functions. Next, following \cite{Pulses} and \cite{Schmitt}, Bacry and Muzy \cite{BM} introduce special
conical\footnote{The reader should note that other conical sets
can be used to construct the measure. The other choices, however, lead to somewhat different properties of the limit measure, cf. \cite{barralID} for a particular example.}
sets $\mathcal{A}_{\varepsilon}(u)$ in the time-scale
plane defined by
\begin{equation}
\mathcal{A}_{\varepsilon}(u) = \left\{(t,l) \,\,\Big\vert\,\,
|t-u|\leq\frac{l}{2} \,\,\text{for}\,\,\varepsilon\leq l\leq 1
\,\,\text{and}\,\,|t-u|\leq\frac{1}{2}\,\,\text{for}\,\,l\geq
1\right\}.
\end{equation}
The sets
$\mathcal{A}_{\varepsilon}(u)$ and
$\mathcal{A}_{\varepsilon}(v)$ intersect iff $|u-v|<1.$ 
It is easy to see that 
the intensity measure of intersections satisfies
\begin{equation}\label{rho}
\rho_{\varepsilon}(|u-v|) \triangleq
\rho\left(\mathcal{A}_{\varepsilon}(u)\bigcap\mathcal{A}_{\varepsilon}(v)\right)
= \begin{cases}
\log(1/|u-v|)& \, \text{if $\varepsilon\leq |u-v|\leq 1$}, \\
1+\log(1/\varepsilon)-|u-v|/\varepsilon & \, \text{if
$|u-v|<\varepsilon$},
\end{cases}
\end{equation}
and it is identically zero for $|u-v|>1.$ It is clear that
$u\rightarrow P\left(\mathcal{A}_{\varepsilon}(u)\right)$
is a stationary, ID process such that 
$P\left(\mathcal{A}_{\varepsilon}(u)\right)$ and $P\left(\mathcal{A}_{\varepsilon}(v)\right)$ are dependent iff $|u-v|<1.$ With probability one, the process
$u\rightarrow P\left(\mathcal{A}_{\varepsilon}(u)\right)$ has right-continuous
trajectories with finite left limits.

Given these preliminaries, the limit logID measure $M_{\mu}(dt)$ on the interval $[0,\,1]$ 
associated with $\phi(q)$ at intermittency $\mu$ 
is the zero scale limit $\varepsilon\rightarrow 0$ of
finite scale random measures that are
defined to be the exponential functional of the $u\rightarrow P\left(\mathcal{A}_{\varepsilon}(u)\right)$ process.
\begin{equation}
M_{\mu}(a, b)=\lim\limits_{\varepsilon\rightarrow 0}\int\limits_a^b \exp\Bigl(P\bigl(\mathcal{A}_{\varepsilon}(u)\bigr)\Bigr) \, du.
\end{equation}
The limit exists in the weak a.s. sense
as was formally established in
\cite{BM1} based on \cite{K2} using the normalization of $\phi(q)$ and the property of $P$ of
being independently scattered.
The limit measure has the stationarity property
\begin{equation}\label{station}
M_{\mu}(t,t+\tau) \overset{{\rm in \,law}}{=} M_{\mu}(0,\tau)
\end{equation}
and is 
nondegenerate in the sense of ${\bf E}[M_{\mu}(a, b)]=|b-a|$
under the assumption\footnote{The nondegeneracy condition given in \cite{BM1} is
less stringent than Eq. \eqref{nondeg}, which is however sufficient in most
cases of interest such as those of the limit lognormal, compound
Poisson, etc. processes.} that
\begin{equation}\label{nondeg}
1+i\mu\phi'(-i) = 1-\mu\Bigl(\frac{\sigma^2}{2} + \int\limits_{\mathbb{R}\setminus \{0\}} \bigl(ue^{u}-e^u+1
\bigr) d\mathcal{M}(u)\Bigr)
>0.
\end{equation}
The moments $q>1$ of $M_{\mu}(0, t)$ are finite under the
following necessary and sufficient conditions
\begin{subequations}
\begin{align}
& q-\mu\phi(-iq)>1 \Longrightarrow {\bf E}[M^q_{\mu}(0, t)]<\infty, \label{finmom} \\
& {\bf E}[M^q_{\mu}(0, t)]<\infty \Longrightarrow
q-\mu\phi(-iq)\geq 1.
\end{align}
\end{subequations}
The combination $q-\mu\phi(-iq)$ is known as
the {\it multiscaling spectrum}. We have
\begin{equation}\label{mspectrum}
q-\mu\phi(-iq) = q-\mu\Bigl(\frac{\sigma^2}{2}(q^2-q) + \int\limits_{\mathbb{R}\setminus \{0\}} \bigl(e^{qu}-1-q(e^u-1)
\bigr) d\mathcal{M}(u)\Bigr).
\end{equation}
Its significance has to do with the remarkable stochastic self-similarity property of the limit measure. Given $t<1,$ let $\Omega_t$
denote an ID random variable that is independent of $ M_{\mu}(0, 1)$ such that
\begin{equation}\label{X}
{\bf E}\left[e^{i q \Omega_t}\right] = e^{-\mu\log t\phi(q)}.
\end{equation}
Then, the property of stochastic self-similarity is
\begin{equation}\label{multifractal}
M_{\mu}(0, t) \overset{{\rm in \,law}}{=} t\exp\bigl(\Omega_t\bigr) M_{\mu}(0, 1),
\end{equation}
understood as the equality of random variables in law at fixed
$t<1.$ It now follows from Eqs. \eqref{X} and \eqref{multifractal} that the moments
obey the multiscaling law for $q$ such that ${\bf E}[M_{\mu}^q(0, 1)] <\infty$
\begin{equation}\label{multiscaling}
{\bf E}\bigl[M_{\mu}^q(0, t)\bigr]=
t^{q-\mu\phi(-iq)}{\bf E}\bigl[M_{\mu}^q(0, 1)\bigr], \,\,\,t<1.
\end{equation}
Hence, $q\rightarrow q-\mu\phi(-iq)$ is the multiscaling spectrum of
the limit measure. 

We conclude our review of the limit logID construction with a
fundamental lemma due to \cite{BM1}.
\begin{lemma}[Main lemma]\label{main}
Given $t_1\leq\cdots\leq t_n$ and $q_1,\cdots, q_n,$ the joint
characteristic function of $P\left(\mathcal{A}_{\varepsilon}(t_j)\right),$
$j=1\cdots n,$ is
\begin{equation}\label{mainchar}
{\bf E}\Biggl[\exp\Bigl(i\sum_{j=1}^n q_j P\bigl(\mathcal{A}_{\varepsilon}(t_j)\bigr)\Bigr)\Biggr] = \exp\Bigl(\mu\sum_{p=1}^n
\sum_{k=1}^p \alpha_{p,k} \,\rho_{\varepsilon}(t_p-t_k)\Bigr),
\end{equation}
where $\rho_\varepsilon(u)$ is defined in Eq. \eqref{rho} and
the coefficients $\alpha_{p,k}$ are given in terms of $\phi(q)$ by
\begin{equation}\label{alpha}
\alpha_{p,k} = \phi(r_{k,p}) + \phi(r_{k+1,p-1}) - \phi(r_{k,p-1}) - \phi(r_{k+1,p}),\; 
\end{equation}
and $r_{k, p} = \sum_{m=k}^p q_m$ if $k\leq p$ and zero otherwise. 
\end{lemma}
This lemma is of crucial significance as it is the principal
computational tool in the study of limit logID measures and, in particular,
implies all the known invariances of the $u\rightarrow P\bigl(\mathcal{A}_{\varepsilon}(u)\bigr)$ process, cf. \cite{Me5} for details.

\section{Exact results on moments of the total mass}
\noindent In this section we will derive a multiple integral representation of the moments of the total mass of the general limit logID measure. Let $M_\mu(dt)$ denote the limit logID measure corresponding to some fixed $\phi(q).$ Throughout this section we assume that the intermittency parameter $\mu$ satisfies \eqref{nondeg} and
the order of the moment  $n\in\mathbb{N}$ satisfies \eqref{finmom}. Also, $f(t)$ denotes a generic non-negative test function to be integrated with respect to the limit measure (the reader can assume $f(t)=1$ with little loss of generality).\footnote{By a slight abuse of terminology, we refer to any integral of the form $\int_0^1 f(t)\,M_\mu(dt)$ as the total mass.} 
\begin{theorem}[Single moments]\label{single}
Given $m\in\mathbb{N},$ let $d(m)$ be defined by
\begin{equation}\label{d}
d(m)\triangleq \sigma^2+\int\limits_{\mathbb{R}\setminus \{0\}} e^{(m-1)u}
(e^u-1)^2\,d\mathcal{M}(u).
\end{equation}
Then, the $n$th moment is given by a generalized Selberg integral of dimension $n$
\begin{equation}\label{singlemomformula}
{\bf E}\Bigl[\Bigl(\int\limits_0^1 f(t)\,M_\mu(dt)\Bigr)^n\Bigr] = n!
\int\limits_{0<t_1<\cdots<t_n<1} \prod\limits_{i=1}^n f(t_i)\prod\limits_{k<p}^n |t_p-t_k|^{-\mu\,d(p-k)}
\,dt.
\end{equation}
\end{theorem}
The same type of result can be formulated for the joint moments. For simplicity of notation, we only state it here for two subintervals of the unit interval
but it should be clear that the following result applies to any finite number of
non-overlapping subintervals.
\begin{theorem}[Joint moments]\label{joint} 
Let $I_j=(a_j,b_j),$ $j=1,2$ such that $I_j\subset (0, 1)$ and $b_1\leq a_2.$ 
Then, the joint $(n,m)$ moment is given by a generalized Selberg integral of dimension $n+m$
\begin{align}
{\bf E}\Bigl[\Bigl(\int\limits_{I_1} f_1(t)\,M_\mu(dt)\Bigr)^{n}\Bigl(\int\limits_{I_2} f_2(t)\,M_\mu(dt)\Bigr)^{m}\Bigr] = n!m! \!\!\!\!\!\!\!\!\!\!\!\!\!\!\!\!
\int\limits_{\substack{a_1<t_1<\cdots<t_{n}<b_1 \\a_2<t_{n+1}<\cdots<t_{n+m}<b_2 }} & \prod\limits_{i=1}^{n} f_1(t_i) \prod\limits_{i=n+1}^{n+m} f_2(t_i)\times \nonumber \\ 
&\!\!\!\!\!\!\!\!\!\!\!\times\prod\limits_{k<p}^{n+m}|t_p-t_k|^{-\mu\,d(p-k)}
\,dt.
\end{align}
\end{theorem}
\begin{corollary}[Covariance structure]\label{cov}
Let $0<t<1.$
\begin{equation}\label{covstruc}
{\bf Cov}\Bigl(\log\int\limits_t^{t+\tau}M_\mu(dt), \,\log\int\limits_0^{\tau} M_\mu(dt)\Bigr) = -\mu\log t\Bigl(\sigma^2 + \int\limits_{\mathbb{R}\setminus \{0\}} u^2 
\,d\mathcal{M}(u)\Bigr) + O(\tau).
\end{equation}
\end{corollary}

Our last two results in this section have to do with the structure of the integrals in Theorems \ref{single} and \ref{joint}, respectively, and are motivated by the goal of formulating a proper generalization of the Selberg integral for an arbitrary limit logID measure. To this end, we make the following definition.
\begin{definition}[Selberg integral for the limit logID measure]
\begin{equation}\label{Sdef}
S_n\bigl(\lambda, \lambda_1,\lambda_2\bigr) \triangleq \int\limits_{0<t_1<\cdots<t_n<1} \prod\limits_{i=1}^n  t_i^{\lambda_1 d(i)}(1-t_i)^{\lambda_2 d(n-i+1)}\prod\limits_{k<p}^n |t_p-t_k|^{2\lambda\,d(p-k)}
\,dt
\end{equation}
\end{definition}
for generally complex $\lambda,$ $\lambda_1,$ and $\lambda_2.$ 
If $\mathcal{M}(u)=0$ and $\sigma=1,$
this definition clearly recovers the classical Selberg integral, cf. Eq. \eqref{selberg} below, and it coincides with the integral in Theorem \ref{single} if
$\lambda=-\mu/2,$ $\lambda_1=\lambda_2=0,$ \emph{i.e.} formally
\begin{equation}
n! S_n(-\mu/2, 0, 0) = {\bf E}\Bigl[\bigl(M_\mu(0, 1)\bigr)^n\Bigr].
\end{equation}
It has the symmetry
\begin{equation}\label{Isymm}
S_n\bigl(\lambda, \lambda_1,\lambda_2\bigr) = S_n\bigl(\lambda, \lambda_2,\lambda_1\bigr),
\end{equation}
which is verified by changing variables $t'_i = 1-t_{n+1-i}.$ 
In addition, one expects the values of this integral at positive, integer $\lambda$ to determine the values for all $\lambda$ as is well-known to be the case for the Selberg integral.
\begin{theorem}[Recurrence relation: single moments]\label{mylemma}
Let $n=2,3,4\cdots$ and $S_0=1.$ Then,
\begin{align}\label{Iinvar} 
S_n(\lambda, 0, 0) = & 
\frac{1}{\bigl(n-1+2\lambda\phi(-in)\bigr)\bigl(n+2\lambda\phi(-in)\bigr)} \int\limits_{0<t_2<\cdots<t_{n-1}<1} \prod_{i<j}^n |t_i-t_j|^{2\lambda d(j-i)}\Big\vert_{\substack{t_1=0 \\ t_n=1}} \,dt, \\
= & \frac{1}{\bigl(n-1+2\lambda\phi(-in)\bigr)\bigl(n+2\lambda\phi(-in)\bigr)}   
S_{n-2}(\lambda, 2\lambda, 2\lambda),
\end{align}
where we have by Eq. \eqref{phi} for $n\in\mathbb{N}$
\begin{equation}\label{phin}
\phi(-in) = \frac{\sigma^2}{2}(n^2-n) + 
\int\limits_{\mathbb{R}\setminus \{0\}} \Bigl(e^{nu}-1-n(e^u-1)
\Bigr) d\mathcal{M}(u).
\end{equation}
\end{theorem}
This result shows that the dependence of the integral on $\lambda_1$ and $\lambda_2$ gives a recurrence relation for the full integral. This property
combined with the symmetry in Eq. \eqref{Isymm} and correct behavior for $\mathcal{M}(u)=0$ suggest that Eq. \eqref{Sdef} is the proper definition
of the Selberg integral corresponding to an arbitrary limit logID measure.
We will apply Theorem \ref{mylemma} in the next section to calculate
low moments of the limit log-Poisson measure. 

A similar results holds for the joint moments, except that there is a term for each pair of boundary points. We will formulate it here in the
case of two subintervals of the unit interval as in Theorem \ref{joint}.
\begin{theorem}[Recurrence relation: joint moments]\label{mylemma2}
Let $N=n+m.$
\begin{gather}
\int\limits_{\substack{a_1<t_1<\cdots<t_{n}<b_1 \\a_2<t_{n+1}<\cdots<t_{N}<b_2 }} \prod\limits_{k<p}^{N}|t_p-t_k|^{2\lambda\,d(p-k)}\,dt
= \frac{1}{\bigl(N-1+2\lambda\phi(-iN)\bigr)\bigl(N+2\lambda\phi(-iN)\bigr)} \times \nonumber \\
\times\Biggl[ 
(b_2-a_1)^2\!\!\!\!\!\!\!\!\!\!\!\!\!\!\!\!\!\!\!\!\!\!\!\!\!\int\limits_{\substack{a_1<t_2<\cdots<t_{n}<b_1 \\a_2<t_{n+1}<\cdots<t_{N-1}<b_2 }} \prod\limits_{k<p}^{N}|t_p-t_k|^{2\lambda\,d(p-k)}\Big\vert_{\substack{t_1=a_1 \\ t_N=b_2}}\,dt
+ (b_1-a_1)^2\!\!\!\!\!\!\!\!\!\!\!\!\!\!\!\!\!\!\!\!\!\!\!\!\!\!\!\int\limits_{\substack{a_1<t_2<\cdots<t_{n-1}<b_1 \\a_2<t_{n+1}<\cdots<t_{N}<b_2 }} \prod\limits_{k<p}^{N}|t_p-t_k|^{2\lambda\,d(p-k)}\Big\vert_{\substack{t_1=a_1 \\ t_n=b_1}}\,dt \nonumber \\
- (a_2-a_1)^2 \!\!\!\!\!\!\!\!\!\!\!\!\!\!\!\!\!\!\!\!\!\!\!\!\!\!\!\int\limits_{\substack{a_1<t_2<\cdots<t_{n}<b_1 \\a_2<t_{n+2}<\cdots<t_{N}<b_2 }} \prod\limits_{k<p}^{N}|t_p-t_k|^{2\lambda\,d(p-k)}\Big\vert_{\substack{t_1=a_1 \\ t_{n+1}=a_2}}\,dt - (b_2-b_1)^2 \!\!\!\!\!\!\!\!\!\!\!\!\!\!\!\!\!\!\!\!\!\!\!\!\!\!\!\int\limits_{\substack{a_1<t_1<\cdots<t_{n-1}<b_1 \\a_2<t_{n+1}<\cdots<t_{N-1}<b_2 }} \prod\limits_{k<p}^{N}|t_p-t_k|^{2\lambda\,d(p-k)}\Big\vert_{\substack{t_{n}=b_1 \\ t_{N}=b_2}}\,dt \nonumber \\
+ (b_2-a_2)^2 \!\!\!\!\!\!\!\!\!\!\!\!\!\!\!\!\!\!\!\!\!\!\!\!\!\!\!\int\limits_{\substack{a_1<t_1<\cdots<t_{n}<b_1 \\a_2<t_{n+2}<\cdots<t_{N-1}<b_2 }} \prod\limits_{k<p}^{N}|t_p-t_k|^{2\lambda\,d(p-k)}\Big\vert_{\substack{t_{n+1}=a_2 \\ t_{N}=b_2}}\,dt
+(a_2-b_1)^2 \!\!\!\!\!\!\!\!\!\!\!\!\!\!\!\!\!\!\!\!\!\!\!\!\!\!\! \int\limits_{\substack{a_1<t_1<\cdots<t_{n-1}<b_1 \\a_2<t_{n+2}<\cdots<t_{N}<b_2 }} \prod\limits_{k<p}^{N}|t_p-t_k|^{2\lambda\,d(p-k)}\Big\vert_{\substack{t_n=b_1 \\ t_{n+1}=a_2}}\,dt
\Biggr]. \label{jointrec}
\end{gather}
\end{theorem}
It is clear that this recurrence relation suggests how one should define
the analogue of Eq. \eqref{Sdef} for the joint moments but we will not attempt to write down the formal definition here. We will give 
an application of Theorem \ref{mylemma2} to joint moments of the limit lognormal measure in Section 5.

Before we give the proofs we will illustrate our results with two principal examples.
\begin{example}[Limit lognormal moments]
Let $\sigma=1$ and $\mathcal
{M}(u)=0$ in Eq. \eqref{phi}. Then,
\begin{equation}\label{singleLlog}
{\bf E}\Bigl[\Bigl(\int\limits_0^1 f(t)\,M_\mu(dt)\Bigr)^n\Bigr] = n!
\int\limits_{0<t_1<\cdots<t_n<1} \prod\limits_{i=1}^n f(t_i)\prod\limits_{k<p}^n |t_p-t_k|^{-\mu}
\,dt.
\end{equation}
This formula was first derived in \cite{BDM} for $f(t)=1$ and extended to arbitrary $f(t)$ in \cite{MeIMRN}.
Note that the nondegeneracy condition in Eq. \eqref{nondeg} amounts to
$0<\mu<2$ and that the moments become infinite for $n>2/\mu.$ In the limit lognormal case
Theorem \ref{joint} for $f(t)=1$ and Corollary \ref{cov} are originally 
due to \cite{MRW}.
\end{example}
\begin{example}[Limit Log-Poisson moments]
Let $\sigma=0$ and $d\mathcal{M}(u) = \delta\bigl(u-log(c)\bigr)du$ in Eq. \eqref{phi}, \emph{i.e.}
the underlying distribution is a point mass at $\log(c),$ $c>0,$ $c\neq 1.$
\begin{equation}\label{P}
{\bf E}\Bigl[\Bigl(\int\limits_0^1 f(t)\,M_\mu(dt)\Bigr)^n\Bigr] = n!
\int\limits_{0<t_1<\cdots<t_n<1} \prod\limits_{i=1}^n f(t_i)\prod\limits_{k<p}^n |t_p-t_k|^{-\mu (c-1)^2 c^{p-k-1}}
\,dt.
\end{equation}
We believe that this formula is new. The nondegeneracy condition in Eq. \eqref{nondeg} is
\begin{equation}\label{nondegP}
0<\mu<\frac{1}{c\log(c)-c+1},
\end{equation}
so that the limit log-Poisson measure exists for any such $c$ as $c\log(c)-c+1>0$
for $c>0,$ $c\neq 1.$ The moments are finite for $q>1$ if 
\begin{equation}\label{momentexistP}
q-\mu\bigl(c^q-1-q(c-1)\bigr)>1 \Longrightarrow {\bf E}\Bigl[M_\mu(0,1)^q\Bigr]<\infty,
\end{equation}
cf. Eqs. \eqref{finmom} and \eqref{mspectrum}.
In particular, the moments become eventually infinite if
$c>1$ as they do in the limit lognormal case. On the contrary, if $0<c<1,$
all moments for $q>1$ are finite for sufficiently small $\mu$ such as, for example, 
\begin{equation}
\mu\leq \frac{1}{1-c} \Longrightarrow {\bf E}\Bigl[M_\mu(0,1)^q\Bigr]<\infty.
\end{equation}
\end{example}

We now proceed to give the proofs. While Theorem \ref{single} is formally a special case of Theorem \ref{joint}, the proof of Theorem \ref{joint} is a straightforward extension of that of Theorem \ref{single} so that we restrict ourselves to the latter for simplicity.
\begin{proof}[Proof of Theorem \ref{single}]
The proof is based on the main lemma, cf. Lemma \ref{main}. We need to compute
${\bf E}\Bigl[\exp\Bigl(\sum_{j=1}^n P\bigl(\mathcal{A}_{\varepsilon}(t_j)\bigr)\Bigr)\Bigr]$ given 
$t_1\leq\cdots\leq t_n,$ which corresponds to $q_j=-i$ in Lemma \ref{main}.
We wish to show that the corresponding coefficients $\alpha_{p, k},$ cf. Eq. \eqref{alpha}, satisfy 
\begin{equation}\label{alphad}
\alpha_{p,k} = 
\begin{cases} d(m) & \,\text{if $m\triangleq p-k>0$}, \\
0 & \, \text{if $k=p$},
\end{cases}
\end{equation} 
where $d(m)$ is defined in Eq. \eqref{d} above.
Let $m=p-k>0.$ Then, 
\begin{align}
\alpha_{p, k} & = \phi\bigl(-i(m+1)\bigr)+\phi\bigl(-i(m-1)\bigr)-2\phi\bigl(-im\bigr), \nonumber \\
& = d(m)
\end{align}
by Eq. \eqref{phin}. If $k=p,$ $\alpha_{p, k}=\phi(-i)=0$ by the normalization of $\phi(q).$
The proof is now completed by a simple limiting procedure. Using Fubini's theorem and the symmetry of the integrand,
\begin{align}
{\bf E}\Bigl[\Bigl(\int\limits_0^1 f(s)\,M_\mu(ds)\Bigr)^n\Bigr] & = \lim\limits_{\varepsilon\rightarrow 0} {\bf E}\Bigl[\Bigl( \int\limits_0^1 f(s)\, e^{P(\mathcal{A}_{\varepsilon}(s))}\,ds\Bigr)^n\Bigr], \nonumber \\
& = n! \lim\limits_{\varepsilon\rightarrow 0}\Biggl[\int\limits_{\{s_1<\cdots< s_n\}} \prod_{r=1}^n f(s_r) \,{\bf E}\Bigl[e^{P(\mathcal{A}_{\varepsilon}(s_{1}))+\cdots +P(\mathcal{A}_{\varepsilon}(s_{n}))}\Bigr]ds\Biggr].
\end{align}
Now, recalling Eqs. \eqref{mainchar} and \eqref{alphad}, we can write
\begin{equation}\label{limint}
{\bf E}\Bigl[\Bigl(\int\limits_0^1 f(s)\,M_\mu(ds)\Bigr)^n\Bigr] = n! \lim\limits_{\varepsilon\rightarrow 0}\Biggl[\int\limits_{\{s_1<\cdots< s_n\}} \prod_{r=1}^n f(s_r) \, \exp\Bigl(\mu\sum_{p=1}^n
\sum_{k=1}^{p-1} d(p-k) \,\rho_{\varepsilon}(s_p-s_k)\Bigr)ds\Biggr].
\end{equation}
It remains to interchange the order of the integration and $\varepsilon$ limit, which is done by 
the dominated convergence theorem. It is convenient to introduce the function 
\begin{equation}
\bar{\rho}_{\varepsilon}(|u-v|) 
\triangleq  \begin{cases}
-\log(|u-v|)& \, \text{if $\varepsilon\leq |u-v|\leq 1$}, \\
-\log(\varepsilon) & \, \text{if $|u-v|<\varepsilon$}.
\end{cases}
\end{equation}
The advantage of $\bar{\rho}_{\varepsilon}$ over $\rho_{\varepsilon}$ in Eq. \eqref{rho} is that it is monotone in $\varepsilon.$
It satisfies the inequality
\begin{equation}
\rho_{\varepsilon}(|u-v|) \leq  \bar{\rho}_{\varepsilon}(|u-v|) + 1. 
\end{equation} 
Denote the integrand in Eq. \eqref{limint} by $g_\varepsilon(s)$ and the function that is defined by replacing $\rho_{\varepsilon}$ with
$\bar{\rho}_{\varepsilon}$ in Eq. \eqref{limint} by $\bar{g}_\varepsilon(s).$ As $d(m)\geq 0,$ we have the inequality 
\begin{equation}\label{gbarineq}
g_\varepsilon(s) \leq e^{\mu \sum_{p=1}^n
\sum_{k=1}^{p-1} d(p-k)}\,\bar{g}_\varepsilon(s).
\end{equation}
We will now show that 
\begin{equation}\label{gbarequat}
\lim\limits_{\varepsilon\rightarrow 0} \int\limits_{\{s_1<\cdots< s_n\}}  \bar{g}_\varepsilon(s) ds = 
\int\limits_{\{s_1<\cdots< s_n\}}  \lim\limits_{\varepsilon\rightarrow 0}  \bar{g}_\varepsilon(s) ds < \infty.
\end{equation}
The function $\bar{g}_\varepsilon(s)$ is non-negative and monotone in $\epsilon$ so that the order of the integral and limit can be interchanged
by the monotone convergence theorem. To prove that the limit is finite, we note that $\bar{g}_\varepsilon(s)$  is large only over the regions where $s_p$ and $s_k$ are close. 
We will estimate the contribution of such a region to the whole integral.
Assume $s_l<\cdots <s_j$ are within $\varepsilon$ apart for some $1\leq l<j\leq n.$ The Lebesgue measure of this region
is of the order $O(\varepsilon^{j-l}).$ The value of the integrand is of the order 
\begin{equation}
\bar{g}_\varepsilon(s) =
O\Bigl(\varepsilon^{-\mu\sum\limits_{p=l}^j
\sum\limits_{k=l}^{p-1} d(p-k)}\Bigr).
\end{equation}
It is easy to see from the definition of $d(m)$ in Eq. \eqref{d} that
\begin{equation}
\sum\limits_{j=1}^n\sum\limits_{l=1}^{j-1} d(j-l) = \phi(-in).
\end{equation}
It follows that the order of $\bar{g}_\varepsilon(s)$ over this region is $O\bigl(\varepsilon^{-\mu \phi(-i(j-l+1)}\bigr).$ Hence, the total contribution of
this region to the integral is of the order $O\bigl(\varepsilon^{j-l - \mu \phi(-i(j-l+1))}\bigr).$
Finally, $j-l$ is at most $n-1$ and there are finitely many such regions so that their
total contribution to the integral is of the order 
\begin{equation}
O\Bigl(\varepsilon^{n-1 - \mu \phi(-in)}\Bigr).
\end{equation}
This gives us the estimate
\begin{equation}\label{estimate}
 \int\limits_{\{s_1<\cdots< s_n\}}  \bar{g}_{\varepsilon/2}(s) \,ds -  \int\limits_{\{s_1<\cdots< s_n\}}  \bar{g}_{\varepsilon}(s) \,ds 
 =  O\Bigl(\varepsilon^{n-1 - \mu \phi(-in)}\Bigr).
\end{equation}
It remains to recall the condition in Eq. \eqref{finmom} for the finiteness of the $n$th moment of the total mass, 
\begin{equation}
n-\mu \phi(-in)>1,
\end{equation}
so that the exponent in Eq. \eqref{estimate} is positive, 
hence the limit in Eq. \eqref{gbarequat} is finite.
The dominated convergence theorem and Eqs. \eqref{gbarineq} and \eqref{gbarequat} then imply 
\begin{equation}
\lim\limits_{\varepsilon\rightarrow 0} \int\limits_{\{s_1<\cdots< s_n\}}  g_\varepsilon(s) ds =
\int\limits_{\{s_1<\cdots< s_n\}}  \lim\limits_{\varepsilon\rightarrow 0}  g_\varepsilon(s) ds <\infty.
\end{equation}
It remains to note that the limit
\begin{equation}
\lim\limits_{\varepsilon\rightarrow 0}  g_\varepsilon(s)  
= \prod\limits_{r=1}^n f(s_r)\prod\limits_{k<p}^n |s_p-s_k|^{-\mu\,d(p-k)},
\end{equation}
coincides with the integrand in Eq. \eqref{singlemomformula}. 
\qed
\end{proof}
\begin{proof}[Proof of Corollary \ref{cov}]
We follow the idea of the proof given in \cite{MRW} in the limit lognormal case. The starting point is the identity that holds for positive random variables $A$ and $B$ 
\begin{equation}\label{AB}
{\bf{Cov}} \left(A^q,\,\,B^q\right)=q^2\,{\bf{Cov}} \left(\log A,\,\log B\right)+o(q^2)
\end{equation}
and follows by Taylor expanding $A^q$ and $B^q$ around $q=0.$
We will apply this identity to $\int_t^{t+\tau}M_\mu(dt)$ and  $\int_0^{\tau} M_\mu(dt)$ by computing the covariance on the left-hand side of Eq. \eqref{AB}
for integer $q=n$ by Theorem \ref{joint} and then analytically continuing to $q\rightarrow 0.$ 
Let $\tau<t$ and $t+\tau<1.$ We have by Theorem \ref{joint}
\begin{equation}
{\bf E}\Bigl[\Bigl(\int\limits_t^{t+\tau} M_\mu(dt)\Bigr)^{n}\Bigl(\int\limits_0^\tau M_\mu(dt)\Bigr)^{n}\Bigr] = (n!)^2 
\int\limits_{\substack{0<t_{1}<\cdots<t_{n}<\tau \\ t<t_{n+1}<\cdots<t_{2n}<t+\tau}} \prod\limits_{k<p}^{2n}|t_p-t_k|^{-\mu\,d(p-k)}
\,dt.
\end{equation}
We will now estimate the magnitude of the cross terms. Clearly,
\begin{equation}
|t-\tau|\leq |t_p-t_k| \leq t+\tau
\end{equation}
for $k=1\cdots n$ and $p=n+1\cdots 2n.$ Hence,
\begin{equation}\label{est}
|t+\tau|^{-\mu\sum\limits_{k=1}^n  \sum\limits_{p=n+1}^{2n} d(p-k) }\leq\prod\limits_{\substack{k=1\cdots n\\p=n+1\cdots 2n}}^{2n}|t_p-t_k|^{-\mu\,d(p-k)} \leq |t-\tau|^{-\mu\sum\limits_{k=1}^n  \sum\limits_{p=n+1}^{2n} d(p-k) }.
\end{equation}
The double sum in Eq. \eqref{est} can be computed exactly using Eq. \eqref{d}.
\begin{equation}
\sum\limits_{k=1}^n  \sum\limits_{p=n+1}^{2n} d(p-k) = \sigma^2 n^2 +
\int\limits_{\mathbb{R}\setminus \{0\}}  e^{(n-1)u} \Bigl[\frac{e^u-e^{(n+1)u}}{1-e^u}\Bigr]\Bigl[
\frac{e^{-u}-e^{-u(n+1)}}{1-e^{-u}}\Bigr]
(e^u-1)^2\,d\mathcal{M}(u). 
\end{equation}
We now observe that the expression on the right-hand side of this equation is analytic in $n,$ 
which allows us to continue it to $n=q\rightarrow 0.$ We then obtain in this limit
\begin{equation}
\sum\limits_{k=1}^n  \sum\limits_{p=n+1}^{2n} d(p-k)\Big\vert_{n=q} = q^2\Bigl(\sigma^2 + \int\limits_{\mathbb{R}\setminus \{0\}} u^2 
\,d\mathcal{M}(u)\Bigr)+o(q^2)\;\text{as $q\rightarrow 0$}.
\end{equation}
It follows that we get the estimates as $q\rightarrow 0$
\begin{align}
{\bf E}\Bigl[\Bigl(\int\limits_t^{t+\tau} M_\mu(dt)\Bigr)^{q}\Bigl(\int\limits_0^\tau M_\mu(dt)\Bigr)^{q}\Bigr] \leq & 
{\bf E}\Bigl[\Bigl(\int\limits_t^{t+\tau} M_\mu(dt)\Bigr)^{q}\Bigr] {\bf E}\Bigl[\Bigl(\int\limits_0^\tau M_\mu(dt)\Bigr)^{q}\Bigr] \times \nonumber \\ &\times \Bigl(1-q^2\mu  \log|t-\tau|
\bigl(\sigma^2 + \int\limits_{\mathbb{R}\setminus \{0\}} u^2 
\,d\mathcal{M}(u)\bigr)\Bigr) + o(q^2),
\end{align}
\begin{align}
{\bf E}\Bigl[\Bigl(\int\limits_t^{t+\tau} M_\mu(dt)\Bigr)^{q}\Bigl(\int\limits_0^\tau M_\mu(dt)\Bigr)^{q}\Bigr] \geq & 
{\bf E}\Bigl[\Bigl(\int\limits_t^{t+\tau} M_\mu(dt)\Bigr)^{q}\Bigr] {\bf E}\Bigl[\Bigl(\int\limits_0^\tau M_\mu(dt)\Bigr)^{q}\Bigr] \times \nonumber \\ &\times\Bigl(1-q^2\mu  \log|t+\tau|
\bigl(\sigma^2 + \int\limits_{\mathbb{R}\setminus \{0\}} u^2 
\,d\mathcal{M}(u)\bigr)\Bigr) + o(q^2).
\end{align}
Substituting them into Eq. \eqref{AB} and taking the limit $q\rightarrow 0,$
we obtain
\begin{align}
{\bf Cov}\Bigl(\log\int\limits_t^{t+\tau} M_\mu(dt),\,\log\int\limits_0^\tau M_\mu(dt)\Bigr) & \leq -\mu \log|t-\tau|
\Bigl(\sigma^2 + \int\limits_{\mathbb{R}\setminus \{0\}} u^2 
\,d\mathcal{M}(u)\Bigr), \\
{\bf Cov}\Bigl(\log\int\limits_t^{t+\tau} M_\mu(dt),\,\log\int\limits_0^\tau M_\mu(dt)\Bigr) & \geq -\mu \log|t+\tau|
\Bigl(\sigma^2 + \int\limits_{\mathbb{R}\setminus \{0\}} u^2 
\,d\mathcal{M}(u)\Bigr),
\end{align}
and the result follows.
\qed
\end{proof}
\begin{proof}[Proof of Theorems \ref{mylemma} and \ref{mylemma2}]
We note first that the definition of $S_n(\lambda,\lambda_1,\lambda_2)$ implies the identity
\begin{equation}
S_{n-2}(\lambda, 2\lambda, 2\lambda) =     
\int\limits_{0<t_2<\cdots<t_{n-1}<1} \prod_{i<j}^n |t_i-t_j|^{2\lambda d(j-i)}\Big\vert_{\substack{t_1=0\\ t_n=1}} \,dt.
\end{equation}
The key element of the proof is the following pair of identities\footnote{We first discovered these identities in the special case of the limit log-Poisson measure using the Almkvist-Zeilberger algorithm as implemented in the Maple package MultiAlmkvistZeilberger.} 
\begin{align}
\sum\limits_{l\neq s}^{n} \frac{\partial}{\partial t_l} \Bigl[ (t_l-t_s)\prod\limits_{i<j}^n |t_i-t_j|^{2\lambda d(j-i)} \Bigr] 
 = & \Bigl(n-1+2\lambda\sum\limits_{i<j}^n d(j-i)\Bigr) \prod\limits_{i<j}^n |t_i-t_j|^{2\lambda d(j-i)},  \label{id1} \\
\sum\limits_{l\neq s}^{n} \frac{\partial}{\partial t_l} \Bigl[ (t_l-t_s)(t_p-t_s)\prod\limits_{i<j}^n |t_i-t_j|^{2\lambda d(j-i)} \Bigr]
= & \Bigl(n+2\lambda\sum\limits_{i<j}^n d(j-i)\Bigr) (t_p-t_s)\prod\limits_{i<j}^n |t_i-t_j|^{2\lambda d(j-i)}, \label{id2}
\end{align}
where $p\neq s$ are any two indices from 1 to $n.$
Assuming Eqs. \eqref{id1} and \eqref{id2}, the proof of Theorem \ref{mylemma} is immediate. Indeed, 
as we already noted in the proof of Theorem \ref{single},
the definition of $d(m)$ in Eq. \eqref{d} implies that
\begin{equation}
\sum\limits_{l<j}^n d(j-l) = \phi(-in).
\end{equation}
We now apply Eq. \eqref{id1} with $s=1$ followed by Eq. \eqref{id2} with
$s=1$ and $p=n$ and notice that all boundary terms but one cancel at each step. 
\begin{align}
S_n(\lambda, 0, 0) = & \frac{1}{\bigl(n-1+2\lambda\phi(-in)\bigr)}
\int\limits_{0<t_2<\cdots<t_{n}<1} t_n\prod_{i<j}^n |t_i-t_j|^{2\lambda d(j-i)}\Big\vert_{t_1=0} dt, \\
=& \frac{1}{\bigl(n-1+2\lambda\phi(-in)\bigr)\bigl(n+2\lambda\phi(-in)\bigr)}
\int\limits_{0<t_2<\cdots<t_{n-1}<1} \prod_{i<j}^n |t_i-t_j|^{2\lambda d(j-i)}\Big\vert_{\substack{t_1=0\\ t_n=1}} \,dt.
\end{align}
The proof of Theorem \ref{mylemma2} is quite similar. We first apply Eq. \eqref{id1} with $s=n,$ for example, which gives us three nonzero boundary terms corresponding to the indices 1, $n+1,$ and $N.$ 
We then apply Eq. \eqref{id2} to each of these terms with ($s=1,$ $p=n$),
($s=n+1,$ $p=n$), and ($s=N,$ $p=n$), respectively, resulting in nine nontrivial boundary terms altogether. The result follows after a straightforward algebraic reduction. 
Finally, Eqs. \eqref{id1} and \eqref{id2} are verified by a direct calculation.
It is easy to see by inspection
\begin{equation}
\frac{1}{\prod\limits_{i<j}^n |t_i-t_j|^{2\lambda d(j-i)}}\sum\limits_{l\neq s}^{n} \frac{\partial}{\partial t_l} \Bigl[ (t_l-t_s)\prod\limits_{i<j}^n |t_i-t_j|^{2\lambda d(j-i)} \Bigr] = n-1 + 2\lambda\sum\limits_{l\neq s}^{n}\sum\limits_{j\neq l}^n d(|j-l|) \frac{t_l-t_s}{t_l-t_j}.
\end{equation}
Finally, separate off $j=s$ and sum over the remaining pairs
\begin{align}
\sum\limits_{l\neq s}^{n}\sum\limits_{j\neq l}^n d(|j-l|) \frac{t_l-t_s}{t_l-t_j} = & \sum\limits_{l\neq s}^{n} d(|s-l|) + \sum\limits_{l<j\neq s}^{n} \Bigl( d(j-l) \frac{t_l-t_s}{t_l-t_j} + d(j-l) \frac{t_j-t_s}{t_j-t_l}\Bigr), \nonumber \\
= & \sum\limits_{l<j}^{n} d(j-l).
\end{align} 
The proof of Eq. \eqref{id2} is very similar and will be omitted. 
\qed
\end{proof}

\section{Calculation of low moments of limit log-Poisson measures}
\noindent In this section we will focus our attention on the limit log-Poisson measure to illustrate the general theory with a concrete nontrivial example 
of a limit logID measure that is different from the limit lognormal measure. 
To this end, we will consider the moments of the limit log-Poisson measure
corresponding to $c=2$ and $c=1/2$ in Eq. \eqref{P}. We assume
that Eqs. \eqref{nondegP} and \eqref{momentexistP} are satisfied in this section. Then, we have by Theorem \ref{single}
and Eq. \eqref{P}
\begin{align}
{\bf{E}} \Bigl[\Bigl(M_\mu^{(c=2)}(0,1)\Bigr)^n\Bigr] & =  n!
\int\limits_{0<t_1<\cdots<t_n<1} \prod\limits_{k<p}^n |t_p-t_k|^{-\mu  2^{p-k-1}}
\,dt, \\ 
{\bf{E}} \Bigl[\Bigl(M_\mu^{(c=1/2)}(0,1)\Bigr)^n\Bigr] & = n!
\int\limits_{0<t_1<\cdots<t_n<1} \prod\limits_{k<p}^n |t_p-t_k|^{-\mu  2^{k-p-1}}
\,dt.
\end{align}
Similarly for the joint moments we have by Theorem \ref{joint}
\begin{equation}
{\bf E}\Bigl[\Bigl(M_\mu^{(c=2)}\bigl(0,\frac{1}{2}\bigr)\Bigr)^{n}\Bigl( M_\mu^{(c=2)}\bigl(\frac{1}{2}, 1\bigr)\Bigr)^{m}\Bigr] = n!m! 
\int\limits_{\substack{0<t_1<\cdots<t_{n}<1/2\\1/2<t_{n+1}<\cdots<t_{n+m}<1 }} \prod\limits_{k<p}^{n+m}|t_p-t_k|^{-\mu\,2^{p-k-1}}
\,dt,
\end{equation}
\begin{equation}
{\bf E}\Bigl[\Bigl(M_\mu^{(c=1/2)}\bigl(0,\frac{1}{2}\bigr)\Bigr)^{n}\Bigl( M_\mu^{(c=1/2)}\bigl(\frac{1}{2}, 1\bigr)\Bigr)^{m}\Bigr] = n!m! 
\int\limits_{\substack{0<t_1<\cdots<t_{n}<1/2\\1/2<t_{n+1}<\cdots<t_{n+m}<1 }} \prod\limits_{k<p}^{n+m}|t_p-t_k|^{-\mu\,2^{k-p-1}}
\,dt.
\end{equation}
The result of this section is the computation of the single moments
up to $n=4$ and of the joint moments for $n=1,$ $m=1,2.$ To streamline notation, we will express our results in terms of $\lambda$ as opposed to $\mu$ as follows. Consider two families of integrals
\begin{align}
I_n\bigl(\lambda\bigr) \triangleq & \int\limits_{0<x_1<\cdots<x_n<1} \prod\limits_{i<j}^n |x_i-x_j|^{\lambda 2^{j-i}} dx = S_n^{(c=2)}(\lambda, 0, 0) = {\bf{E}} \Bigl[\Bigl(M_\mu^{(c=2)}(0,1)\Bigr)^n\Bigr]/n!, \\
J_n\bigl(\lambda\bigr) \triangleq & \int\limits_{0<x_1<\cdots<x_n<1}  \prod\limits_{i<j}^n |x_i-x_j|^{\lambda 2^{n-(j-i)}} dx = S_n^{(c=1/2)}(2^n\lambda, 0, 0) = {\bf{E}} \Bigl[\Bigl(M_\mu^{(c=1/2)}(0,1)\Bigr)^n\Bigr]/n!,
\end{align}
corresponding to special cases of Eq. \eqref{Sdef} with $d(m)=2^{m-1}$ ($c=2,$ $\lambda=-\mu/2,$ $\lambda_1=\lambda_2=0$) and $d(m)=2^{-m-1}$ ($c=1/2,$ $\lambda=-2^{-n}\mu/2,$ $\lambda_1=\lambda_2=0$), respectively. 
\begin{proposition}[Single moments]\label{Poissonmom}
\begin{align}
I_2(\lambda) = & J_2(\lambda) = \frac{1}{(1+2\lambda)(2+2\lambda)}, \label{P1}\\
I_3(\lambda) = & \frac{1}{(2+8\lambda)(3+8\lambda)} \frac{\Gamma(1+2\lambda)^2}{\Gamma(2+4\lambda)}, \label{P2}\\
I_4(\lambda) = & 
\frac{1}{(3+22\lambda)(4+22\lambda)}  \frac{\Gamma(1+2\lambda)^2}{\Gamma(2+4\lambda)} \frac{\Gamma(1+2\lambda)\Gamma(2+8\lambda)}{\Gamma(3+10\lambda)}\times \nonumber \\ &\times {}_3 F_2(-4\lambda, \,1+2\lambda, \, 2+8\lambda; \,2+4\lambda,\,3+10\lambda; \,1), \label{P3} \\
J_3(\lambda) = & \frac{1}{(2+10\lambda)(3+10\lambda)} \frac{\Gamma(1+4\lambda)^2}{\Gamma(2+8\lambda)}, \label{P4} \\
J_4(\lambda) = &  \frac{1}{(3+34\lambda)(4+34\lambda)}  \frac{\Gamma(1+8\lambda)^2}{\Gamma(2+16\lambda)}  \frac{\Gamma(1+8\lambda)\Gamma(2+20\lambda)}{\Gamma(3+28\lambda)}\times \nonumber \\
& \times {}_3  F_2(-4\lambda, \,1+8\lambda, \, 2+20\lambda; \,2+16\lambda,\,3+28\lambda; \,1). \label{P5}
\end{align}
\end{proposition}
The corresponding joint moments can now be expressed in terms of the single moments as follows. 
\begin{corollary}[Joint moments]\label{Poissonjoint}
\begin{align}
{\bf E}\Bigl[M_\mu^{(c)}\bigl(0,\frac{1}{2}\bigr)\,M_\mu^{(c)}\bigl(\frac{1}{2}, 1\bigr)\Bigr] = & \frac{1}{2}\bigl(1-2^{\mu(c-1)^2-1}\bigr) {\bf E}\Bigl[\Bigl(M_\mu^{(c)}(0,1)\Bigr)^2\Bigr], \\
{\bf E}\Bigl[M_\mu^{(c)}\bigl(0,\frac{1}{2}\bigr)\Bigl(M_\mu^{(c)}\bigl(\frac{1}{2}, 1\bigr)\Bigr)^2\Bigr] = & \frac{1}{6}\bigl(1-2^{\mu(c^3-3c+2)-2}\bigr){\bf E}\Bigl[\Bigl(M_\mu^{(c)}(0,1)\Bigr)^3\Bigr].
\end{align}
\end{corollary}
\begin{proof}[Proof of Proposition \ref{Poissonmom}]
Theorem \ref{mylemma} enables us to reduce the dimension of the corresponding integral by two, resulting in simpler integrals, which can be computed by elementary means. In fact, Eq. \eqref{P1} is a special case of Eq. \eqref{Sdef}. The integral for the 3rd moment is reduced to the product of two beta integrals, resulting in Eq. \eqref{P2}. The integral for the 4th moment 
is
\begin{align}
I_4(\lambda) = &\frac{1}{(3+22\lambda)(4+22\lambda)}  
\int\limits_{0<x_2<x_3<1} x_2^{2\lambda} x_3^{4\lambda} (1-x_2)^{4\lambda}(1-x_3)^{2\lambda} (x_3-x_2)^{2\lambda} dx, \nonumber \\
= & \frac{1}{(3+22\lambda)(4+22\lambda)}  \int\limits_{[0,1]^2}
y^{2\lambda}(1-y)^{2\lambda}z^{8\lambda+1}(1-z)^{2\lambda}(1-yz)^{4\lambda} dydz.
\end{align}
The resulting integral is a special case of Lemma 2 of \cite{Nesterenko}.
The proof for $J_n(\lambda)$ is the same.
\qed
\end{proof} 
\begin{proof}[Proof of Corollary \ref{Poissonjoint}]
This result follows from the general identity
\begin{equation}\label{sumid}
{\bf E}\Bigl[M_\mu(0,1)^n\Bigr] = \sum\limits_{k=0}^n \binom{n}{k} {\bf E}\Bigl[M_\mu\bigl(0,\frac{1}{2}\bigr)^{k}\,M_\mu\bigl(\frac{1}{2}, 1\bigr)^{n-k}\Bigr]
\end{equation}
and the multiscaling law of the limit measure, cf. Eq. \eqref{multiscaling}.
\qed
\end{proof}

\section{Joint limit lognormal moments}
\noindent In this section we will continue our study of joint moments of the limit lognormal measure that we began in \cite{Me6}, where we established
the limit lognormal analogue of Corollary \ref{Poissonjoint}. The corresponding multiple integrals have no known closed-form formula similar to Selberg's, cf. Chapter 4 of \cite{ForresterBook}, 
\begin{equation}\label{selberg}
\int\limits_{[0,1]^N} \prod\limits_{i=1}^N t_i^{\lambda_1}(1-t_i)^{\lambda_2}\prod\limits_{k<p}^N |t_p-t_k|^{2\lambda}
\,dt
 = \prod_{k=0}^{N-1}
\frac{\Gamma(1+(k+1)\lambda)\Gamma(1+\lambda_1+k\lambda)\Gamma(1+\lambda_2+k\lambda)}
{\Gamma(1+\lambda)\Gamma(2+\lambda_1+\lambda_2+(N+k-1)\lambda)},
\end{equation}
which gives the single moments, see Eq. \eqref{singleLlog} ($\lambda=-\mu/2$). 
As we will show in this section, the principal difficulty of computing joint moments
is that they involve non-standard hypergeometric-like integrals that do not appear to have the same type of
remarkable cancellations as those that occur in the Selberg integral. The challenge of evaluating these integrals 
poses the problem of finding an alternative representation for them that further reveals their structure.
With this goal in mind, in this section we will present a combinatorial re-formulation of these integrals. 
The interest in a combinatorial approach to the joint moments is based on the recent success of combinatorial
techniques in the context of the classical Selberg integral, cf. \cite{Petrov}.
The main result of this section is a multiple binomial sum representation of the joint moments. 
In addition, the same method gives a similar representation of the single moments, \emph{i.e.} the 
Selberg integral with $\lambda_1=\lambda_2=0$, and of the Morris integral, which we believe are also new. 
In particular, the comparison of the binomial sum representations of the single and joint moments
reveals the source of cancellations in the Selberg integral that are lacking in the joint moments.


The starting point is the formula for the joint moments, cf. Theorem \ref{joint}.
Let $N=n+m.$
\begin{align}
{\bf E}\Bigl[&\Bigl(M_\mu\bigl(0,\frac{1}{2}\bigr)\Bigr)^{n}\Bigl( M_\mu\bigl(\frac{1}{2}, 1\bigr)\Bigr)^{m}\Bigr] =  \int\limits_{[0, \,1/2]^{n}} 
\int\limits_{[1/2, \,1]^{m}} \prod\limits_{k<l}^{N}
|x_k-x_l|^{-\mu} dx, \\
 = &
2^{N(N-1)\mu/2-N}\int\limits_{[0,1]^n} \prod_{k<l}^{n} |y_k-y_l|^{-\mu} \prod_{\substack{k=1\cdots n \\
l=n+1\cdots N}} |y_k+y_l|^{-\mu} \times\prod_{n+1\leq
k<l}^{N} |y_k-y_l|^{-\mu} dy,
\end{align}
where we changed variables $y_k=1-2x_k,$ $k=1\cdots n,$ and $y_k=2x_k-1,$
$k=n+1\cdots N.$
Hence, we will consider the pair of integrals that are parameterized by $N\in\mathbb{N},$ $N=n+m.$ 
\begin{align}
S_N(\lambda)  = & \int\limits_{[0,1]^N} \prod_{k<l}^{N} |y_k-y_l|^{2\lambda} dy
 = \prod_{j=0}^{N-1}
\frac{\Gamma(1+(j+1)\lambda)\Gamma(1+j\lambda)^2}
{\Gamma(1+\lambda)\Gamma(2+(N+j-1)\lambda)},
\label{s1} \\
S_{n, m}(\lambda) = & \int\limits_{[0,1]^N} \prod_{k<l}^{n} |y_k-y_l|^{2\lambda} \prod_{\substack{k=1\cdots n \\
l=n+1\cdots N}} |y_k+y_l|^{2\lambda} \prod_{n+1\leq
k<l}^{N} |y_k-y_l|^{2\lambda} dy.\label{s2}
\end{align}
The integrals for the joint moments are not known in closed from, except for the relation
\begin{equation}\label{sumL}
2^{\lambda N(N-1)+N} S_N(\lambda) = \sum\limits_{n+m=N} \binom{N}{n} S_{n, m}(\lambda),
\end{equation}
which is a special case of Eq. \eqref{sumid}, that determines $S_{1,1}(\lambda)$ and $S_{1,2}(\lambda).$
For example, the simplest nontrivial joint moments are
$S_{1, 3}(\lambda)$ and $S_{2, 2}(\lambda).$ 
We have by Eq. \eqref{s2}
\begin{align}
S_{1, 3}(\lambda) = & \int\limits_0^1 \int\limits_0^1 \int\limits_0^1 \int\limits_0^1
|x_1+x_2|^{2\lambda} |x_1+x_3|^{2\lambda} |x_1+x_4|^{2\lambda} 
|x_2-x_3|^{2\lambda} |x_2-x_4|^{2\lambda} |x_3-x_4|^{2\lambda}\,dx, \\
S_{2, 2}(\lambda) = & \int\limits_0^1 \int\limits_0^1 \int\limits_0^1 \int\limits_0^1
|x_1-x_2|^{2\lambda} |x_1+x_3|^{2\lambda} |x_1+x_4|^{2\lambda} 
|x_2+x_3|^{2\lambda} |x_2+x_4|^{2\lambda} |x_3-x_4|^{2\lambda}\,dx,
\end{align}
and the relation that follows from Eq. \eqref{sumL}
\begin{equation}
4 S_{1,3}(\lambda)+3 S_{2,2}(\lambda) = \bigl(2^{12\lambda+3}-1\bigr) S_4(\lambda).
\end{equation}
The principle challenge of computing them is the occurrence of `+' signs in the integrands. This is best illustrated by
reducing them to 2-dimensional integrals by means of Theorem \ref{mylemma2}. 
\begin{proposition}[$S_{1,3}$ and $S_{2,2}$ as 2-dimensional integrals]\label{twoD}
\begin{align}
S_{1, 3}(\lambda) = & \frac{1}{4(3\lambda+1)(4\lambda+1)}  \int\limits_0^1 \int\limits_0^1  \Bigl[2\bigl[(1-x)^2 x^2 (1+y)^2 y^2 (x+y)^2 \bigr]^\lambda 
 -  \bigl[(1+x)^2 x^2 (1+y)^2 y^2 (x-y)^2\bigr]^\lambda 
 \nonumber \\  &  + 4^{\lambda+1}
 \bigl[(1-x)^2 (1+x)^2 (1+y)^2 (1-y)^2 (x-y)^2\bigr]^\lambda
 -  \bigl[(1-x)^2 x^2 (1-y)^2 y^2 (x-y)^2\bigr]^\lambda \Bigr]\,dxdy,\label{S13} \\
S_{2, 2}(\lambda) = & \frac{1}{3(3\lambda+1)(4\lambda+1)}  \int\limits_0^1 \int\limits_0^1  \Bigl[ \bigl[(1+x)^2 x^2 (1+y)^2 y^2 (x-y)^2\bigr]^\lambda 
 -  2\bigl[(1-x)^2 x^2 (1+y)^2 y^2 (x+y)^2 \bigr]^\lambda 
 \nonumber \\  &  + 4^{\lambda+1}
 \big[ (1-x)^2 (1+x)^2 (1+y)^2 (1-y)^2 (x+y)^2 \bigr]^\lambda \Bigr]\,dxdy,\label{S22}
\end{align}
\end{proposition}
\begin{proof}
This is a corollary of Theorem \ref{mylemma2}. In our case $\phi(-in)=n(n-1)2,$
$a_1=0,$ $b_1=a_2=1/2,$ $b_2=1,$ and $d(m)= 1.$ The result follows from Eq. \eqref{jointrec}
by elementary changes of variables $x\rightarrow 1-2x$ and $x\rightarrow 2x-1$
that map $[0, 1/2]$ and $[1/2, 1]$ onto $[0,1],$ respectively.
\qed
\end{proof}
The resulting integrals are of a non-standard hypergeometric form precisely due to the presence of `+' signs and are 
beyond our computational reach. 

We will now restrict ourselves to $\lambda\in\mathbb{N}.$
The rationale for this condition is the well-known fact that the values of $S_N(\lambda)$ at $\lambda\in\mathbb{N}$ determine its values for all $\lambda.$ One expects the same property of $S_{n, m}(\lambda)$ so that we can restrict ourselves to $\lambda\in\mathbb{N}$ without loss of generality. 
On the other hand, this restriction allows us to interpret these integrals in the form of binomial sums and thereby reveal
their combinatorial structure. 

The integrands in Eqs. \eqref{s1} and \eqref{s2} are products over $N(N-1)/2$ pairs of variables $|y_k\pm y_l|^{2\lambda}.$ For each pair $k<l$ let $I_{kl},$
denote an index of summation that runs over $-\lambda\cdots\lambda,$ 
and define $I_{kl}=-I_{lk}$ for $k>l.$ Let
$s_{kl}$ be the indicator function of whether the sign of the $(k, l)$ pair is negative, \emph{i.e.}
\begin{equation}\label{s}
s_{kl} =
\begin{cases}
1& \, \text{if the pair is $|y_k - y_l|^{2\lambda}$}, \\
0 & \, \text{if
the pair is $|y_k + y_l|^{2\lambda}$}.
\end{cases}
\end{equation}
Finally, let $\sum\limits_{\substack{I_{kl}=-\lambda\\k<l=1\cdots N}}^\lambda$ denote the sum over all indices $I_{kl},$ $k<l,$ $k, l=1\cdots N,$ such that $I_{kl}\in[-\lambda,\lambda].$
\begin{proposition}[Binomial sums for single and joint moments]\label{ssums}
\begin{align}
S_N(\lambda) = & (-1)^{\lambda \frac{N(N-1)}{2}} \sum\limits_{\substack{I_{kl}=-\lambda\\k<l=1\cdots N}}^\lambda
(-1)^{\sum\limits_{k<l} I_{kl}}
\prod\limits_{k<l}^N  \binom{2\lambda}{\lambda+I_{kl}}\prod\limits_{k=1}^N
\frac{1}{1+(N-1)\lambda +\sum\limits_{l\neq k} I_{kl}}, \label{selbergsum}\\
S_{n, m}(\lambda) = & (-1)^{\lambda\bigl(n(n-1)+ m(m-1)\bigr)/2} \sum\limits_{\substack{I_{kl}=-\lambda\\k<l=1\cdots N}}^\lambda
(-1)^{\sum\limits_{k<l} s_{kl} I_{kl}}
\prod\limits_{k<l}^N  \binom{2\lambda}{\lambda+I_{kl}}\prod\limits_{k=1}^N
\frac{1}{1+(N-1)\lambda +\sum\limits_{l\neq k} I_{kl}}.\label{snmsum}
\end{align}
\end{proposition}
It must be pointed out that the expressions in Eqs. \eqref{selbergsum} and \eqref{snmsum} are remarkably similar. 
The principle difference is in the prefactors $ (-1)^{\sum\limits_{k<l} I_{kl}} $ and $(-1)^{\sum\limits_{k<l} s_{kl} I_{kl}},$
the latter involving fewer indices of summation than the former as $s_{kl}$ can be zero so that the sum in \eqref{snmsum}
has fewer cancellation. We believe that it is this lack of cancellations that gives rise to the additional complexity of
the joint moments.

We mention is passing that we have a similar result for the Morris integral, cf. Chapters 3 and 4 of \cite{ForresterBook},
which describes the total mass of the limit lognormal measure on the circle, cf. \cite{FyoBou} and \cite{Me16}. 
Let $a, b\in \mathbb{N}.$
\begin{align}
M_N(a, b, \lambda) &\triangleq \int\limits_{[-\frac{1}{2},\,\frac{1}{2}]^N} \prod\limits_{l=1}^N  e^{ \pi i \theta_l(a-b)}  |1+e^{ 2\pi i\theta_l}|^{a+b} \,
\prod\limits_{k<l}^N |e^{ 2\pi i \theta_k}-e^{2\pi i\theta_l}|^{2\lambda} \,d\theta, \\ &=\prod\limits_{j=0}^{N-1} \frac{\Gamma(1+a+b+\lambda j)\,\Gamma(1+\lambda (j+1))}{\Gamma(1+a+\lambda j)\,\Gamma(1+b+\lambda j)\,\Gamma(1+\lambda)} \label{morris2}.
\end{align}
\begin{proposition}[Binomial sum for the Morris integral]\label{morrissum}
\begin{equation}
M_N(a, b, \lambda) =  \sum\limits_{\substack{I_{kl}=-\lambda\\k<l=1\cdots N}}^\lambda
(-1)^{\sum\limits_{k<l} I_{kl}}
\prod\limits_{k<l}^N  \binom{2\lambda}{\lambda+I_{kl}}\prod\limits_{k=1}^N
\binom{a+b}{a+\sum\limits_{l\neq k} I_{kl}}.
\end{equation}
\end{proposition}

These results are most easily explained through examples. Let $N=3.$
\begin{align}
S_{1,2}(\lambda)  = & (-1)^\lambda \sum\limits_{I_{12}=-\lambda}^\lambda  \sum\limits_{I_{13}=-\lambda}^\lambda 
 \sum\limits_{I_{23}=-\lambda}^\lambda (-1)^{I_{23}}\binom{2\lambda}{\lambda+I_{12}}\binom{2\lambda}{\lambda+I_{13}}\binom{2\lambda}{\lambda+I_{23}} \times \nonumber \\
& \times \frac{1}{1+2\lambda+I_{12}+I_{13}}\frac{1}{1+2\lambda-I_{12}+I_{23}} \frac{1}{1+2\lambda-I_{13}-I_{23}}, \\
S_{3}(\lambda)  = & (-1)^\lambda \sum\limits_{I_{12}=-\lambda}^\lambda  \sum\limits_{I_{13}=-\lambda}^\lambda 
 \sum\limits_{I_{23}=-\lambda}^\lambda (-1)^{I_{12}+I_{13}+I_{23}}\binom{2\lambda}{\lambda+I_{12}}\binom{2\lambda}{\lambda+I_{13}}\binom{2\lambda}{\lambda+I_{23}} \times \nonumber \\
& \times \frac{1}{1+2\lambda+I_{12}+I_{13}}\frac{1}{1+2\lambda-I_{12}+I_{23}} \frac{1}{1+2\lambda-I_{13}-I_{23}}, \label{ssums3} \\
M_3(a, b, \lambda) = & \sum\limits_{I_{12}=-\lambda}^\lambda  \sum\limits_{I_{13}=-\lambda}^\lambda 
 \sum\limits_{I_{23}=-\lambda}^\lambda (-1)^{I_{12}+I_{13}+I_{23}}\binom{2\lambda}{\lambda+I_{12}}\binom{2\lambda}{\lambda+I_{13}}\binom{2\lambda}{\lambda+I_{23}} 
\times \nonumber \\
& \times \binom{a+b}{a+I_{12}+I_{13}}\binom{a+b}{a-I_{12}+I_{23}} \binom{a+b}{a-I_{13}-I_{23}}. \label{morris3}
\end{align}
One can similarly write down multiple binomial sum expressions for $S_{1,3}(\lambda)$ and
$S_{2,2}(\lambda),$ which involve 6-dimensional sums. 

We will sketch proofs of our summation formulas for $N=3$ only for simplicity
as the general case requires extensive book-keeping that would take us too far afield. 
\begin{proof}[Proof of Propositions  \ref{ssums} and \ref{morrissum}]
Consider $S_{1,2}(\lambda).$ By Eq. \eqref{s2} we have
\begin{equation}
S_{1, 2}(\lambda) = \int\limits_{[0,1]^3} |y_1+y_2|^{2\lambda}  |y_1+y_3|^{2\lambda} |y_2-y_3|^{2\lambda} dy.
\end{equation}
Expanding each of the three factors by the binomial formula and evaluating
the resulting integral, we obtain
\begin{equation}
S_{1, 2}(\lambda) =  \sum\limits_{i=0}^{2\lambda}\sum\limits_{j=0}^{2\lambda} 
 \sum\limits_{k=0}^{2\lambda} (-1)^k \binom{2\lambda}{i}\binom{2\lambda}{j}\binom{2\lambda}{k}  \frac{1}{1+i+j}\frac{1}{1+2\lambda-i+k} \frac{1}{1+4\lambda-j-k.}
\end{equation}
The result follows by re-labeling $I_{12}=i-\lambda,$ $I_{13}=j-\lambda,$ $I_{23}=k-\lambda.$ The proof of Eq. \eqref{ssums3} is exactly the same.
The proof of Eq. \eqref{morris3} is similar. One starts with the well-known constant term
representation of the Morris integral, cf. Chapter 3 of \cite{ForresterBook}, 
\begin{equation}
M_N(a, b, \lambda) \triangleq \text{CT}_{\{t_1\cdots t_N\}} \prod\limits_{j=1}^N 
(1-t_j)^a(1-1/t_j)^b \prod\limits_{k<l}^N \Bigl(1-\frac{t_k}{t_l}\Bigr)^\lambda
\Bigl(1-\frac{t_l}{t_k}\Bigr)^\lambda, \label{CTmorris}
\end{equation}
and applies the binomial formula to each factor. The details are straightforward but more involved
and will be omitted.  \qed
\end{proof}


\section{Conclusions}
We have examined the single and joint moments of the total mass of the general limit logID measure of Bacry and Muzy and represented them in the form of
novel Selberg-like integrals involving the L\'evy Khinchine decomposition
of the underlying infinitely divisible distribution. In particular, our formula for the joint moments implies that the covariance structure of the mass of the measure is logarithmic,
thereby extending to the general case what was long known for the limit lognormal measure. We have derived recurrence relations for the single moment integral and the integral corresponding to joint moments of two subintervals. 
Based on the functional form of the recurrence relation, we have formulated a multiple integral that we believe plays the same role for the general limit logID measure as the classical Selberg integral does for the limit lognormal measure. Like the classical integral, the new integral is parameterized
by $(\lambda, \lambda_1, \lambda_2),$ is symmetric in $(\lambda_1, \lambda_2),$
and coincides with the classical integral if the underlying distribution is gaussian. We have illustrated our results with the special case of the limit log-Poisson measure and calculated low moments of its total mass exactly. In the limit lognormal case, we have represented single and joint moments of the total mass in the form of novel multiple binomial sums, resulting in new interpretations of the Selberg integral with $\lambda_1=\lambda_2=0$
and of the Morris integral as binomial coefficient identities. In particular, the multiple sum formulas reveal the source of cancellations, 
which are present in the Selberg integral and are lacking in the integrals for the joint moments, and which are likely responsible for
the complexity of the joint moments. 

The computation of arbitrary single moments of the total mass of a limit logID measure other than the limit lognormal measure and the computation of the joint moments of the latter remain a challenge. We believe that the emerging
structure is hypergeometric, which can already be seen in the functional form
of low limit log-Poisson moments. We have made an attempt to quantify the structure of joint limit lognormal moments in the simplest nontrivial case of 
${\bf E}\bigl[M_\mu(0,1/2)\bigl( M_\mu(1/2,1)\bigr)^{3}\bigr]$
and ${\bf E}\bigl[\bigl(M_\mu(0,1/2)\bigr)^{2}\bigl( M_\mu(1/2,1)\bigr)^{2}\bigr]$
by giving a new representation for these 4-dimensional integrals
in the form of a linear combination of 2-dimensional integrals involving non-standard hypergeometric-like integrands.
The structure of the resulting integrals 
appears to be deep and is left to further research.


\renewcommand{\theequation}{A-\arabic{equation}}

\section*{Acknowledgements}
The author wishes to thank Jakob Ablinger of the Research Institute for Symbolic Computation for verifying Eqs. \eqref{P1} -- \eqref{P5} and attempting to compute the integrals in Eqs. \eqref{S13} and \eqref{S22} using an extension of the Almkvist-Zeilberger algorithm in his Mathematica package MultiIntegrate. The author gratefully acknowledges that he used Doron Zeilberger's implementation of the same algorithm in the Maple package MultiAlmkvistZeilberger. We verified all of our results in Sections 4 and 5 numerically using the computer algebra program MAXIMA (http://maxima.sourceforge.net). Finally, the author is thankful to the referee for several helpful suggestions.

\end{document}